\documentclass{amsart}%
\usepackage{amsfonts}
\usepackage{amsmath}
\usepackage{amssymb}
\usepackage{graphicx}%
\providecommand{\U}[1]{\protect\rule{.1in}{.1in}}
\newtheorem{theorem}{Theorem}

\newtheorem{lemma}[theorem]{Lemma}

\newtheorem{remark}[theorem]{Remark}

\begin{document}

\title{Conformal geometry and special holonomy}
\author{Siu-Cheong Lau and Naichung Conan Leung}

\begin{abstract}
A theorem of Lawson and Simons states that the only stable
minimal submanifolds in $\mathbb{CP}^{n}$ are complex submanifolds.  We generalize their result to the cases of $\mathbb{HP}^{n}$ and $\mathbb{OP}^{2}$.  Our approach gives a unified viewpoint towards conformal and projective geometries.
\end{abstract}

\maketitle

\section{Introduction}

Riemannian holonomy group $hol\left(M,g\right)$ measures the richness of
algebraic structure on a Riemannian manifold\footnote{All manifolds are
connected compact oriented smooth manifolds.}. For a generic metric, the
holonomy group equals $SO\left(  m\right)  $ with $m=\dim_{\mathbb{R}}M$.
Manifolds with special holonomy include K\"{a}hler manifolds with $hol\left(
M,g\right)  =U\left(  n\right)  $ and Calabi-Yau manifolds with \newline%
$hol\left(  M,g\right)  =SU\left(  n\right)  $ where $m=2n$. They play very
important roles in geometry and mathematical physics such as string theory and
M-theory. Riemannian holonomy groups were completely classified by Berger
\cite{Berger_holonomy} and all these geometries have been given a unified
description in terms of real, complex, quaternionic and octonionic structures
(that is, normed division algebras) and orientability in \cite{boss_Huang} for
symmetric spaces and \cite{boss_div_alg} for non-symmetric ones.

Another important branch in Riemannian geometry is the conformal geometry
where one allows the Riemannian metric to be scaled by a conformal factor,
i.e. $g\sim e^{u}g$ for any function $u$. In this article, we explain how one
integrates conformal geometry with real, complex, quaternionic and octonionic
geometries. In particular we give a uniform proof to the following theorem on
rigidity of calibrated cycles in projective spaces, which is a generalization
of the results of Lawson and Simons from conformal and complex geometries to
quaternionic and octonionic geometries. After we have discovered this, we were
informed that this result has been proved earlier by \cite{Ohnita_stable}.
\ We hope that our approach from Jordan algebra provides a unified viewpoint
on all these seemingly different kinds of geometries.

\begin{quotation}
\noindent\textbf{Main Theorem:} In $\mathbb{AP}^{n}$, where $\mathbb{A\in
}\left\{  \mathbb{R},\mathbb{C},\mathbb{H},\mathbb{O},\mathbb{R}^{m}\right\}
$, any stable minimal submanifold $S$ (or more generally rectifiable current)
must be complex, by which we means $T_{x}S$ is invariant under all the linear
complex structures at $x$ for almost every $x\in S$.
\end{quotation}

\begin{remark}
There is an $\mathbb{S}^{2}$-family of linear complex structures at every
point of $\mathbb{H}\mathbb{P}^{n}$, and also an $\mathbb{S}^{6}$-family
of linear complex structures at each point of $\mathbb{O}\mathbb{P}^{2}$.
\end{remark}

\begin{remark}
When $\mathbb{A=O}$, we only allow $n\leq2$; When $\mathbb{A=R}^{m}$, only
$n\mathbb{=}1$ is admitted, and $\mathbb{R}^{m}\mathbb{P}^{1}=\mathbb{S}^{m}$.
\ We will explain this notation in the next section.
\end{remark}

\section{$\mathbb{R}$, $\mathbb{C}$, $\mathbb{H}$, $\mathbb{O}$ and conformal
geometry}

In \cite{boss_div_alg} the second author gave a unified description of
geometries of each holonomy group by first defining the group $G_{\mathbb{A}
}\left(  n\right)  $ of twisted automorphisms of $\mathbb{A}^{n}$ and its
subgroup $H_{\mathbb{A}}\left(  n\right)  $ of special twisted automorphisms,
where \newline$\mathbb{A\in}\left\{  \mathbb{R},\mathbb{C},\mathbb{H}
,\mathbb{O}\right\}  $ is a normed division algebra and $n$ equals one when
$\mathbb{A}=\mathbb{O}$. They are given explicitly in the following table:

\vspace{3mm}%

\begin{tabular}
[c]{|r||c|c|c|c|}\hline
$\mathbb{A}$ & $\mathbb{R}$ & $\mathbb{C}$ & $\mathbb{H}$ & $\mathbb{O}
$\\\hline
$G_{\mathbb{A}}\left(  n\right)  $ & $O\left(  n\right)  $ & $U\left(
n\right)  $ & $Sp\left(  n\right)  Sp\left(  1\right)  $ & $\mathrm{Spin}
\left(  7\right)  $\\\hline
$H_{\mathbb{A}}\left(  n\right)  $ & $SO\left(  n\right)  $ & $SU\left(
n\right)  $ & $Sp\left(  n\right)  $ & $G_{2}$\\\hline
\end{tabular}

\vspace{3mm}

Their corresponding geometries are as follows.

\vspace{3mm}%

\begin{tabular}
[c]{|r||c|c|c|c|}\hline
$\mathbb{A}$ & $\mathbb{R}$ & $\mathbb{C}$ & $\mathbb{H}$ & $\mathbb{O}
$\\\hline
$G_{\mathbb{A}}\left(  n\right)  $ & Riemannian & K\"{a}hler &
Quaternionic-K\"{a}hler & $\mathrm{Spin}\left(  7\right)  $\\\hline
$H_{\mathbb{A}}\left(  n\right)  $ & Volume & Calabi-Yau & Hyperk\"{a}hler &
$G_{2}$\\\hline
\end{tabular}

\vspace{3mm}

Due to the nonassociativity of the octonion, there are obvious difficulties to
define its modules $\mathbb{O}^{n}$ and their automorphism groups
$H_{\mathbb{O}}\left(  n\right)  $. Nonetheless, for $n\leq3$, this problem
can be resolved by considering the space of self-adjoint operators, leading to
the notion of Jordan algebra which we shall describe below.

On $\mathbb{R}^{n}$, the space of self-adjoint operators is simply the space
of symmetric $n\times n$ matrices, denoted by $S_{n}\left(  \mathbb{R}\right)
$. The symmetrization of ordinary matrix multiplication
\[
A\circ B=\left(  AB+BA\right)  /2
\]
makes $S_{n}\left(  \mathbb{R}\right)  $ into a formally real Jordan algebra.
Namely it is an algebra over $\mathbb{R}$ whose multiplication $\circ$ is
commutative and power associative (that is, $(a\circ a)\circ a=a\circ(a\circ
a)$), together with
\[
a_{1}\circ a_{1}+\ldots+a_{n}\circ a_{n}=0\ \Rightarrow\ a_{1}=\ldots
=a_{n}=0\text{.}
\]

The same product also makes the space $S_{n}\left(  \mathbb{A}\right)  $ of
Hermitian symmetric matrices with entries in $\mathbb{A\in}\left\{
\mathbb{R},\mathbb{C},\mathbb{H}\right\}  $ into a Jordan algebra. When $n=3$,
an analog of the product can still be defined for $\mathbb{A}=\mathbb{O}$,
making $S_{3}\left(  \mathbb{O}\right)  $ into an \textit{exceptional Jordan
algebra} (see e.g. \cite{Baez_octonions}) even though $\mathbb{O}$ lacks of associativity.

Inside $S_{n}\left(  \mathbb{A}\right)  $ we may collect all rank one
projections, which are matrices $p$ with $p\circ p=p$ and $\mathrm{tr}\,p=1$,
to form the projective space $\mathbb{AP}^{n-1}$. For instance, while the
module $\mathbb{O}^{3}$ does not exist, the concept of octonion lines
in $\mathbb{O}^{3}$ can be replaced by rank one
projection operators in $S_{3}\left(  \mathbb{O}\right)$, and the space of them forms the \textit{octonion
projective plane }$\mathbb{OP}^{2}$, which can be identified as the symmetric
space $F_{4}/Spin\left(  9\right)  $.

Since $S_{n}\left(  \mathbb{A}\right)  $ and $\mathbb{AP}^{n-1}$ are spaces of
self-adjoint operators on $\mathbb{A}^{n}$, they should share the same
automorphism group $H_{\mathbb{A}}\left(  n\right)  $ as $\mathbb{A}^{n}$.
\ This is indeed true in the classical cases when $\mathbb{A}\in\left\{
\mathbb{R},\mathbb{C},\mathbb{H}\right\}  $ and continues to have such an
interpretation in the exceptional case $\mathbb{A}=\mathbb{O}$. The following
gives a complete list of simple formally real Jordan algebras
\cite{Jordan_alg} and their automorphism groups (The center has removed for simplicity):

\vspace{3mm}%

\begin{tabular}
[c]{|r||c|c|c|c|c|}\hline
$\mathbb{A}$ & $\mathbb{R}$ & $\mathbb{C}$ & $\mathbb{H}$ & $\mathbb{O}$ &
$\mathbb{R}^{m}$\\\hline
$S_{n}\left(  \mathbb{A}\right)  $ & $S_{n}\left(  \mathbb{R}\right)  $ &
$S_{n}\left(  \mathbb{C}\right)  $ & $S_{n}\left(  \mathbb{H}\right)  $ &
$S_{3}\left(  \mathbb{O}\right)  $ & $S_{2}\left(  \mathbb{R}^{m}\right)
\simeq\mathbb{R}^{m}\oplus\mathbb{R}^{1,1}$\\\hline
$\mathbb{AP}^{n-1}$ & $\mathbb{RP}^{n-1}$ & $\mathbb{CP}^{n-1}$ &
$\mathbb{HP}^{n-1}$ & $\mathbb{OP}^{2}$ & $\mathbb{AP}^{1}=\mathbb{S}^{m}
$\\\hline
$H_{\mathbb{A}}\left(  n\right)  $ & $SO\left(  n\right)  $ & $SU\left(
n\right)  $ & $Sp\left(  n\right)  $ & $F_{4}$ & $SO\left(  m+1\right)
$\\\hline
\end{tabular}

\vspace{3mm}

Amazingly there is one more item in the list of Jordan algebras besides those
coming from normed division algebras, namely the \textit{spin factor}
$S_{2}\left(  \mathbb{R}^{m}\right)  \simeq\mathbb{R}^{m}\oplus\mathbb{R}%
^{1,1}$. It consists of $2\times2$ matrices of the form
\[
\left(
\begin{array}
[c]{cc}%
a-b & v\\
v & a+b
\end{array}
\right)  \leftrightarrow\left(
\begin{array}
[c]{c}%
v\\
b\\
a
\end{array}
\right)
\]
where $v\in\mathbb{R}^{m}$ and $a,b\in\mathbb{R}$, and we set $v\cdot
w=v^{t}w$ for $v,w\in\mathbb{R}^{m}$ to carry out matrix multiplication. The
embedded projective space is\newline%
\[
\left\{  \left(
\begin{array}
[c]{c}%
v\\
b\\
\frac{1}{2}%
\end{array}
\right)  :\left\Vert v\right\Vert ^{2}+b^{2}=\frac{1}{4}\right\}
\cong\mathbb{S}^{m}\newline\text{.}%
\]

Notice that the automorphism group $SO\left(  m+1\right)  $ of $S_{2}\left(
\mathbb{R}^{m}\right)  $ is also the isometry group of $\mathbb{S}^{m}$, and
it is contained as a maximal compact subgroup in the non-compact group
$\mathrm{Conf}(\mathbb{S}^{m})=SO\left(  m+1,1\right)  $. A natural question
arises: For $\mathbb{A\in}\left\{  \mathbb{R},\mathbb{C},\mathbb{H}%
,\mathbb{O}\right\}  $, is there a symmetry group of $\mathbb{AP}^{n-1}$ which
gives an analog to the conformal symmetry $SO\left(  m+1,1\right)  $ of
$\mathbb{S}^{m}$?

To answer this question, one identifies $\mathbb{S}^{m}$ as the conformal
boundary of the hyperbolic ball
\[
B^{m+1}:=\{M\in S_{2}\left(  \mathbb{R}^{m}\right)  :\mathrm{\det}M=1\} \cong
SO(m+1,1)/SO(m+1)
\]
on which $SO\left(  m+1,1\right)  $ acts as isometries. Under this
identification, one has $\mathrm{Conf}(\mathbb{S}^{m})\cong\mathrm{Isom}
(B^{m+1})=SO\left(  m+1,1\right)  $ which preserves collinearity in the sense
that $\mathrm{Conf}(\mathbb{S}^{m})$ maps circles to circles in $\mathbb{S}
^{m}$.

Now for $\mathbb{A}\in\left\{  \mathbb{R},\mathbb{C},\mathbb{H}\right\}  $, if
we collect the symmetries of $\mathbb{AP}^{n-1}$ which is linear but not
necessarily isometries, we obtain the group $SL\left(  n,\mathbb{A}\right)  $
\cite{Salzmann_proj_planes}. Analogously $\mathbb{AP}^{n-1}$ can be identified
as a part of the conformal boundary of $\{M\in S_{n}\left(  \mathbb{A}\right)
:\mathrm{\det}M=1\}\cong$ $SL\left(  n,\mathbb{A}\right)  /SU(n,\mathbb{A)}$
on which $SL\left(  n,\mathbb{A}\right)  ~$acts as isometries. We get the
answer for $\mathbb{A}\in\left\{  \mathbb{R},\mathbb{C},\mathbb{H}\right\}  $:
$SL\left(  n,\mathbb{A}\right)  $ can be regarded as the \textit{conformal
symmetry }of $\mathbb{AP}^{n-1}$, which plays the same role as $SO\left(
m+1,1\right)  $ acting on $\mathbb{S}^{m}$. In general, let's denote these
non-compact symmetry groups as $N_{\mathbb{A}}\left(  n\right)  $ which are
listed below. Notice that $H_{\mathbb{A}}\left(  n\right)  $ sits inside
$N_{\mathbb{A}}\left(  n\right)  $ as a maximal compact subgroup, and
$N_{\mathbb{A}}\left(  n\right)  /H_{\mathbb{A}}\left(  n\right)  $ can be
identified with the space of symmetric matrices with determinant one.

\vspace{3mm}%

\begin{tabular}
[c]{|r||c|c|c|c|c|}\hline
$\mathbb{A}$ & $\mathbb{R}$ & $\mathbb{C}$ & $\mathbb{H}$ & $\mathbb{O}$ &
$\mathbb{R}^{m}$\\\hline
$H_{\mathbb{A}}\left(  n\right)  $ & $SO\left(  n\right)  $ & $SU\left(
n\right)  $ & $Sp\left(  n\right)  $ & $F_{4}$ & $SO\left(  m+1\right)
$\\\hline
$N_{\mathbb{A}}\left(  n\right)  $ & $SL\left(  n,\mathbb{R}\right)  $ &
$SL\left(  n,\mathbb{C}\right)  $ & $SL\left(  n,\mathbb{H}\right)  $ &
$E_{6}^{-26}$ & $SO\left(  m+1,1\right)  $\\\hline
\end{tabular}

\vspace{3mm}

We may observe that when $m=1,2,4$ and $8$, $N_{\mathbb{R}^{m}}\left(
2\right)  =SL\left(  2,\mathbb{A}\right)  $ with $\mathbb{A}$ being real,
complex, quaternion and octonion respectively. Hence, $sl\left(
2,\mathbb{R}\right)  =so\left(  2,1\right)  $, $sl\left(  2,\mathbb{C}\right)
=so\left(  3,1\right)  $, $sl\left(  2,\mathbb{H}\right)  =so\left(
5,1\right)  $, $sl\left(  2,\mathbb{O}\right)  =so\left(  9,1\right)  $. In
general we have $sl\left(  2,\mathbb{A}\right)  =so\left(  \mathbb{A}
\oplus\mathbb{R}^{1,1}\right)  $ \cite{Baez_octonions}.

The above point of view integrates conformal geometry with real, complex,
quaternionic and octonionic geometries. In the next section we will illustrate
this viewpoint by studying the variation of volume of cycles under the
conformal symmetry $N_{\mathbb{A}}\left(  n\right)  $ of $\mathbb{AP}^{n-1}$
in a unified manner.

\begin{remark}
In \cite{Atiyah_Proj}, Atiyah and Berndt studied the complexified version of
$\mathbb{AP}^{n-1}$ with $\mathbb{A\in}\left\{  \mathbb{R},\mathbb{C}%
,\mathbb{H},\mathbb{O}\right\}  $. \ We can extend these descriptions to
$\mathbb{A=R}^{m}$ as in the following table:

\vspace{3mm}

{\small
\noindent\begin{tabular}{|r||c|c|c|c|c|}\hline
$\mathbb{A}$ & $\mathbb{R}$ & $\mathbb{C}$ & $\mathbb{H}$ & $\mathbb{O}$ &
$\mathbb{R}^{m}$\\\hline
$\left(  \mathbb{A}\otimes\mathbb{C}\right)  \mathbb{P}^{n-1}$ &
$\mathbb{CP}^{n-1}$ & $\left(  \mathbb{CP}^{n-1}\right)  ^{2}$ &
$Gr_{\mathbb{C}}\left(  2,2n-2\right)  $ & $\frac{E_{6}}{Spin\left(
10\right)  U\left(  1\right)  }$ & $\frac{O\left(  m+2\right)  }{O\left(
m\right)  O\left(  2\right)  }$\\\hline
$H_{\mathbb{A\otimes C}}\left(  n\right)  $ & $SU\left(  n\right)  $ &
$SU\left(  n\right)  ^{2}$ & $SU\left(  2n\right)  $ & $E_{6}$ & $SO\left(
m+2\right)  $\\\hline
$N_{\mathbb{A}\otimes\mathbb{C}}\left(  n\right)  $ & $Sp\left(
2n,\mathbb{R}\right)  $ & $SU\left(  n,n\right)  $ & $O^{\ast}\left(
4n\right)  $ & $E_{7}^{-25}$ & $SO\left(  m+2,2\right)  $\\\hline
\end{tabular}
}

\vspace{3mm}

Notice that the maximal compact subgroup of $N_{\mathbb{A}\otimes\mathbb{C}%
}\left(  n\right)  $ is the product of $H_{\mathbb{A\otimes C}}\left(
n\right)  $ with $U(1)$. \ Furthermore,
\[
\frac{N_{\mathbb{A}\otimes\mathbb{C}}\left(  n\right)  }{H_{\mathbb{A\otimes
C}}\left(  n\right)  U(1)}=S_{n}^{+}\left(  \mathbb{A}\right)  +iS_{n}%
(\mathbb{A)}%
\]
is a tube domain (see for example \cite{Gross_tube}). \ This gives a complete
list of tube domains.

They also have a quaternionic analog:

\vspace{3mm}

{\small
\noindent\begin{tabular}
[c]{|r||c|c|c|c|c|}\hline
$\mathbb{A}$ & $\mathbb{R}$ & $\mathbb{C}$ & $\mathbb{H}$ & $\mathbb{O}$ &
$\mathbb{R}^{m}$\\\hline
$\left(  \mathbb{A}\otimes\mathbb{H}\right)  \mathbb{P}^{n-1}$ &
$\mathbb{HP}^{n-1}$ & $Gr_{\mathbb{C}}(2,2n-2)$ & $Gr_{\mathbb{R}}(4,4n-4)$ &
$\frac{E_{7}}{\mathrm{Spin}(12)O(4)}$ & $\frac{O(m+4)}{O(m)O(4)}$\\\hline
$H_{\mathbb{A\otimes H}}\left(  n\right)  $ & $Sp(n)$ & $SU(2n)$ & $SO(4n)$ &
$E_{7}$ & $SO(m+4)$\\\hline
$N_{\mathbb{A}\otimes\mathbb{H}}\left(  n\right)  $ & $Sp(n,1)$ & $SU(2n,1)$ &
$SO(4n,4)$ & $E_{8}^{-24}$ & $SO(m+4,4)$\\\hline
\end{tabular}
}
\end{remark}

\section{Cycles under conformal symmetries}

In the last section, we regard $N_{\mathbb{A}}\left(  n+1\right)  $ as the
conformal symmetry group of $\mathbb{AP}^{n}$. Its Lie algebra%

\[
\mathfrak{n}_{\mathbb{A}}\left(  n+1\right)  =\mathfrak{h}_{\mathbb{A}}\left(
n+1\right)  \oplus S_{n+1}^{\prime}\left(  \mathbb{A}\right)
\]
induces vector fields which acts infinitestimally on $\mathbb{AP}^{n}$. Here
the Lie algebra $\mathfrak{h}_{\mathbb{A}}\left(  n+1\right)  $ of
$H_{\mathbb{A}}\left(  n+1\right)  $ induces Killing vector fields, and
$S_{n+1}^{\prime}\left(  \mathbb{A}\right)  $ consists of trace-free symmetric
matrices, which can be regarded as constant vector fields in $S_{n+1}^{\prime
}\left(  \mathbb{A}\right)  $, projecting to conformal vector fields on
$\mathbb{AP}^{n}\subset S_{n+1}^{\prime}(\mathbb{A})$. We are adopting the metric%

\[
\left\langle A,B\right\rangle :=2\,\mathrm{Re}(\mathrm{tr}\,AB)=2\,\mathrm{Re}
(\mathrm{tr}\,A\circ B)
\]
on $S_{n+1}^{\prime}(\mathbb{A})$ which induces the standard metric on
$\mathbb{AP}^{n}$.

We would like to compute the average second variation of the volume of a cycle
in $\mathbb{AP}^{n}$ under the action of $\mathfrak{n}_{\mathbb{A}}\left(
n+1\right)  $. First, Let us quickly review the terminology and set up some notations.

\subsection{Terminology and notations}

For a global vector field $V$ on a Riemannian manifold $M$, the second
variation $\mathcal{Q}_{S}(V)$ of the volume $\mathbf{M}$ of a rectifiable
current $S$ under $V$ is defined as%

\[
\mathcal{Q}_{S}(V):=\left.  \frac{\mathrm{d}^{2}}{\mathrm{d}t^{2}}\right\vert
_{t=0}\mathbf{M}((\phi_{t})_{\ast}S)=\int_{M}\,\left.  \frac{\mathrm{d}^{2}
}{\mathrm{d}t^{2}}\right\vert _{t=0}||(\phi_{t})_{\ast}S_{x}||\mathrm{d}
\nu_{S}(x)
\]
where $\phi_{t}$ is the flow induced by $V$, $S_{x}$ denotes the unit simple
vector representing the oriented tangent space of $S$ at $x$, and $\nu_{S}$
denotes the Borel measure associated with $S$. $S$ is said to be stable if
$\mathcal{Q}_{S}(V)\leq0$ for all vector fields $V$ on $M$. We will denote the
integrand $\left.  \frac{\mathrm{d}^{2}}{\mathrm{d}t^{2}}\right\vert
_{t=0}||(\phi_{t})_{\ast}\xi||$ by $\mathcal{Q}_{\xi}\left(  V\right)  $, the
second variation of an oriented orthonormal $p$-frame $\xi$ under $V$. One has
the following second variation formula for a gradient vector field $V$
\cite{Lawson_stable}:%

\begin{align}
\label{sec_var}\mathcal{Q}_{\xi}\left(  V\right)   &  =\left\langle
\mathcal{A}_{V,V}\xi,\xi\right\rangle +2\Vert\mathcal{A}_{V}\xi\Vert
^{2}-(\left\langle \mathcal{A}_{V}\xi,\xi\right\rangle )^{2}\nonumber\\
&  =\left(  \sum_{j=1}^{p}\left\langle \mathcal{A}_{V}e_{j},e_{j}\right\rangle
\right)  ^{2}+2\sum_{j=1}^{p}{\sum\limits_{k=1}^{q}}\left(  \left\langle
\mathcal{A}_{V}e_{j},n_{k}\right\rangle \right)  ^{2}+\sum_{j=1}
^{p}\left\langle \mathcal{A}_{V,V}e_{j},e_{j}\right\rangle
\end{align}
where $\xi=e_{1}\wedge\ldots\wedge e_{p}$, which is extended to an orthonormal
basis \newline$\{e_{1},\ldots,e_{p},n_{1},\ldots,n_{q}\}$ of $TM$. Here for
any smooth vector fields $V$ and $W$, $\mathcal{A}_{V}(u)$, $\mathcal{A}
_{V,W}$ are endomorphisms of $TM$ defined by
\begin{align}
\mathcal{A}_{V}X:=  &  \nabla_{X}V;\nonumber\label{D^2}\\
\mathcal{A}_{V,W}X:=  &  (\nabla_{V}\mathcal{A}_{W})X=\nabla_{V}\nabla
_{\tilde{X}}W-\nabla_{\nabla_{V}\tilde{X}}W
\end{align}
where $\nabla$ is the Levi-Civita connection, $\tilde{X}$ is a smooth local
extension of $X\in TM$. An endomorphism $L$ of $TM$ is extended to operate on
$\bigwedge^{p}TM$ by Leibniz rule:
\[
L(e_{1}\wedge\ldots\wedge e_{p})=\sum_{j=1}^{p}e_{1}\wedge\ldots\wedge
Le_{j}\wedge\ldots\wedge e_{p}\text{.}
\]

From the above second variation formula, we see that $\mathcal{Q}_{\xi}$, and
hence $\mathcal{Q}_{S}$, is a quadratic form on the space of smooth vector
fields on $M$, and we may restrict it to a finite-dimensional subspace $F$ of
vector fields and take the trace $(\mathrm{tr}\,\mathcal{Q}_{\xi}|_{F}
)=\sum\mathcal{Q}_{\xi}(V)$, where $V$ runs through an orthonormal basis of
$F$.

\subsection{Main theorem}

Coming back to our situation $M=\mathbb{AP}^{n}$, since vector fields induced
by $\mathfrak{h}_{\mathbb{A}}\left(  n+1\right)  $ preserve metric and does
not contribute to the second variation, we have
\[
\mathrm{tr}\,\mathcal{Q}_{\xi}|_{\mathfrak{n}_{\mathbb{A}}\left(  n+1\right)
}=\mathrm{tr}\,\mathcal{Q} _{\xi}|_{S_{n+1}^{\prime}\left(  \mathbb{A}\right)
}
\]
and so we may concentrate on $F=S_{n+1}^{\prime}\left(  \mathbb{A}\right)  $.

Moreover, notice that $\mathbb{AP}^{n}$ is an orbit of the group
$H_{\mathbb{A}}\left(  n+1\right)  $ acting on $S_{n+1}^{\prime}\left(
\mathbb{A}\right)  $. This symmetry helps to reduce a lot of calculations, as
illustrated by the following lemma:

\begin{lemma}
\label{inv of secvar under G} Let $G$ act isometrically on an inner product
space $\mathbb{V}$, and $M\subset\mathbb{V}$ be a $G$-invariant submanifold.
The projection of each $u\in\mathbb{V}$ gives a vector field $V_{u}$ on $M$,
and the space of all these vector fields is denoted by $F$. Then
\[
\mathrm{tr}\,\mathcal{Q}_{\xi}|_{F}=\mathrm{tr}\,\mathcal{Q}_{g\cdot\xi}|_{F}
\]
for all $g\in G$.

\begin{proof}
Since the metric on $M$ is $G$-invariant, the Levi-Civita connection $\nabla$
is $G$-equivariant, that is,
\[
\nabla_{g_{\ast}\cdot X}(g_{\ast}\cdot V)=g_{\ast}\cdot(\nabla_{X}V)\text{.}
\]
Hence one has
\[
\mathcal{A}_{V}(g\cdot\xi)=g\cdot(\mathcal{A}_{g_{\ast}^{-1}V}\cdot
\xi);\,\mathcal{A}_{V,W}(g\cdot\xi)=g\cdot(\mathcal{A}_{g_{\ast}
^{-1}V,\,g_{\ast}^{-1}W}\cdot\xi)\text{.}
\]
Applying to the second variation formula, we get
\[
\mathcal{Q}_{g\cdot\xi}(V_{u})=\mathcal{Q}_{\xi}(g_{\ast}^{-1}V_{u}
)=\mathcal{Q}_{\xi}(V_{g_{\ast}^{-1}u})
\]
where the last equality is due to $G$-invariance of metric. And so
\[
\mathrm{tr}\,\mathcal{Q}_{\eta}=\sum_{u}\mathcal{Q}_{\eta}(u)=\sum
_{u}\mathcal{Q}_{\xi}(g_{\ast}^{-1}u)=\mathrm{tr}\,\mathcal{Q}_{\xi}
\]
where $u$, and hence $g_{\ast}^{-1}u$, runs through an orthonormal basis of
$\mathbb{V}$. The last equality follows from the fact that trace is
independent of choice of orthonormal basis.
\end{proof}
\end{lemma}

By the above lemma, where we take $M=\mathbb{AP}^{n},\mathbb{V=}
S_{n+1}^{\prime}\left(  \mathbb{A}\right)  $ and $G=H_{\mathbb{A}}\left(
n+1\right)  $, it suffices to consider average second variation of a $p$-frame
$\xi=e_{1}\wedge\ldots\wedge e_{p}$ at a particular point $x\in\mathbb{AP}
^{n}$, because $p$-frames at another point can be moved to $x$ by some $g\in
H_{\mathbb{A}}\left(  n+1\right)  $. Let's fix $x=\mathbb{E}_{n+1,n+1}
\in\mathbb{AP}^{n}$, which is the matrix with value $1$ at the $(n+1,n+1)$
position and all other entries zero.

We shall need the following formula, whose proof is given in the appendix:

\begin{theorem}
\label{secvarorbit} Assume that $M=G/K\subset$ $\mathbb{V}$ is a compact
symmetric space which is a $G$-orbit of an orthogonal representation
$\mathbb{V}$ of $G$. The projection of each $u\in\mathbb{V}$ gives a vector
field $V_{u}$ on $M$. The average second variation of an oriented orthonormal
$p$-frame $\xi=e_{1}\wedge\ldots\wedge e_{p}$ at $x\in M$ under all such
vector fields is given by
\[
\mathrm{tr}\,\mathcal{Q}_{\xi}=\sum_{j,k=1}^{p,q}\left(  2\,\Vert
\,\mathrm{I}\mathrm{I}(e_{j},n_{k})\Vert^{2}-\left\langle \,\mathrm{I}
\mathrm{I}(e_{j},e_{j}),\,\mathrm{I}\mathrm{I}(n_{k},n_{k})\right\rangle
\right)
\]
\newline where $\mathrm{I}\mathrm{I}$ is the second fundamental form of
$M\subset\mathbb{V}$ at $x$, and $\{e_{j}\}_{j=1}^{p}\cup\{n_{k}\}_{k=1}^{q}$
is an orthonormal basis of $TM$.
\end{theorem}

With the above formula, it remains to compute the second fundamental form of
$\mathbb{AP}^{n}$. Let's take the following coordinates around $x$:
\begin{align*}
\mathbb{A}^{n}  &  \rightarrow\mathbb{A}\mathbb{P}^{n}\subset S_{n+1}^{\prime
}(\mathbb{A})\\
Q  &  \mapsto\frac{1}{1+\Vert Q\Vert^{2}}\left(
\begin{array}
[c]{c}%
Q\\
1
\end{array}
\right)  \left(
\begin{array}
[c]{cc}%
Q^{\ast} & 1
\end{array}
\right)
\end{align*}

Here we adopt the following notations:
\[
Q=\sum_{l=0}^{\Lambda}\mathbf{i}_{l}X_{l}
\]
where $X_{l}$ are column $n$-vectors, $\mathbf{i}_{0}:=1$, and for $1\leq
l\leq\Lambda$, $\mathbf{i}_{l}$ are the linearly independent imaginary square
roots of unity in $\mathbb{A}$. Recall that for the case $\mathbb{A}
=\mathbb{R}^{m}$, $n=1$, $\Lambda=0$, $Q=X_{0}$ is an element in
$\mathbb{R}^{m}$ with $Q^{\ast}=Q$ and $Q\cdot Q:=\left\langle
Q,Q\right\rangle $. For the other four cases, the entries of $X_{l}$ are real numbers.

The basis of coordinate tangent vector fields is $\{\frac{\partial}{\partial
x^{j}_{l}}: 0 \leq l \leq\Lambda, 1 \leq j \leq N\}$, where $\frac{\partial
}{\partial x^{j}_{l}}$ denote the $\mathbf{i}_{l}$-directions. $N = m$ in the
case of $\mathbb{A} = \mathbb{R}^{m}$, and $N = n$ for all the other four
cases. Using product rule (which is valid for multiplication in $\mathbb{A}$),%

\begin{align*}
\left.  \frac{\partial}{\partial x_{l}^{j}}\right\vert _{Q}  &  =\frac
{1}{1+\Vert Q\Vert^{2}}\left(
\begin{array}
[c]{c}%
\mathbf{i}_{l}w_{j}\\
0
\end{array}
\right)  \left(
\begin{array}
[c]{cc}%
Q^{\ast} & 1
\end{array}
\right) \\
&  +\frac{1}{1+\Vert Q\Vert^{2}}\left(
\begin{array}
[c]{c}%
Q\\
1
\end{array}
\right)  \left(
\begin{array}
[c]{cc}%
\overline{\mathbf{i}_{l}}w_{j}^{T} & 0
\end{array}
\right) \\
&  -\frac{2X_{l}^{T}w_{j}}{(1+\Vert Q\Vert)^{2}}\left(
\begin{array}
[c]{c}%
Q\\
1
\end{array}
\right)  \left(
\begin{array}
[c]{cc}%
Q^{\ast} & 1
\end{array}
\right)
\end{align*}
where $w_{j}$ stands for the column $n$-vector with $j$-th coordinate $1$ and
other coordinates zero, and $T$ stands for transpose. Recall that when
$\mathbb{A}=\mathbb{R}^{m}$, $n$ equals $1$, and so transpose of an element is
just itself. Differentiating both sides along $\frac{\partial}{\partial
x_{r}^{k}}$ at $0\in\mathbb{A}^{n}$,
\begin{align*}
&  \left.  \frac{\partial}{\partial x_{r}^{k}}\right\vert _{0}\left(
\frac{\partial}{\partial x_{l}^{j}}\right) \\
&  =\left\{
\begin{array}
[c]{ll}%
\left(
\begin{array}
[c]{cc}%
2\delta_{jk} & 0\\
0 & -2\delta_{jk}%
\end{array}
\right)  & \text{ for $\mathbb{A}=\mathbb{R}^{m}$}\\
& \\
\left(
\begin{array}
[c]{cc}%
\mathbf{i}_{r}\overline{\mathbf{i}_{l}}\mathbb{E}_{kj}+\mathbf{i}_{l}
\overline{\mathbf{i}_{r}}\mathbb{E}_{jk} & 0\\
0 & -(\mathbf{i}_{r}\overline{\mathbf{i}_{l}}+\mathbf{i}_{l}\overline
{\mathbf{i}_{r}})\delta_{jk}%
\end{array}
\right)  & \text{ for }\mathbb{A}=\mathbb{R},\mathbb{C},\mathbb{H},\mathbb{O}%
\end{array}
\right.
\end{align*}
which is already perpendicular to $T_{x}\mathbb{A}\mathbb{P}^{n}$, because
\[
\left.  \frac{\partial}{\partial x_{l}^{j}}\right\vert _{0}=\left(
\begin{array}
[c]{cc}%
0 & \mathbf{i}_{l}w_{j}\\
\overline{\mathbf{i}_{l}}w_{j}^{T} & 0
\end{array}
\right)  \text{.}
\]
Under the metric $\left\langle A,B\right\rangle =2\,\mathrm{Re}\,\mathrm{tr}
\,(AB)$, our coordinate vectors are pairwise orthogonal, each has length $2$.
We scale them to get an orthonormal basis $\{\frac{1}{2}\frac{\partial
}{\partial x_{l}^{j}}:1\leq j\leq n,0\leq l\leq\Lambda\}$.

We conclude that

\begin{lemma}
\label{secform}The second fundamental form $\mathrm{I}\mathrm{I}(\frac{1}
{2}\frac{\partial}{\partial x_{l}^{j}},\frac{1}{2}\frac{\partial}{\partial
x_{r}^{k}})$ of $\mathbb{AP}^{n}\subset S_{n+1}^{\prime}(\mathbb{A})$ at $x$
is given by
\[
\left\{
\begin{array}
[c]{ll}%
\frac{1}{2}\left(
\begin{array}
[c]{cc}%
\delta_{jk} & 0\\
0 & -\delta_{jk}%
\end{array}
\right)  & \text{ for $\mathbb{A}=\mathbb{R}^{m}$}\\
\frac{1}{4}\left(
\begin{array}
[c]{cc}%
\mathbf{i}_{r}\overline{\mathbf{i}_{l}}\mathbb{E}_{kj}+\mathbf{i}_{l}
\overline{\mathbf{i}_{r}}\mathbb{E}_{jk} & 0\\
0 & -(\mathbf{i}_{r}\overline{\mathbf{i}_{l}}+\mathbf{i}_{l}\overline
{\mathbf{i}_{r}})\delta_{jk}%
\end{array}
\right)  & \text{ for }\mathbb{A}\in\{\mathbb{R},\mathbb{C},\mathbb{H}
,\mathbb{O}\}\text{.}%
\end{array}
\right.
\]

\end{lemma}

Now we are ready to compute $\mathrm{tr}\,\mathcal{Q}_{\xi}$ for an
orthonormal $p$-frame \newline$\xi=e_{1}\wedge\ldots\wedge e_{p}$ at
$x\in\mathbb{A}\mathbb{P}^{n}$. Complete $B=\{e_{j}\}_{j=1}^{p}$ to an
orthonormal basis $\{e_{j},n_{k}\}$ in the form
\[
\left\{
\begin{array}
[c]{cccc}%
v_{1}, & \mathbb{J}_{1}v_{1}, & \ldots & \mathbb{J}_{\Lambda}v_{1}\\
\vdots & \vdots &  & \vdots\\
v_{N}, & \mathbb{J}_{1}v_{N}, & \ldots & \mathbb{J}_{\Lambda}v_{N}%
\end{array}
\right\}
\]
where $\mathbb{J}_{l}:T_{x}\mathbb{A}\mathbb{P}^{n}\rightarrow T_{x}
\mathbb{A}\mathbb{P}^{n}$ is the differential of left multiplication of
$\mathbf{i}_{l}$ on $\mathbb{A}^{n}\subset\mathbb{A}\mathbb{P}^{n}$.

Such an orthonormal basis can be brought to the basis of normalized coordinate
vectors by the action of the isotropy group $K<G$. This is easy for
$\mathbb{R}\mathbb{P}^{n}$, $\mathbb{C}\mathbb{P}^{n}$ and $\mathbb{H}%
\mathbb{P}^{n}$: $SO(n)$, $SU(n)$ and $Sp(n)$ acts transitively on orthonormal
frames, unitary frames and quaternionic unitary frames respectively. For
$\mathbb{O}\mathbb{P}^{2}$, $K=\mathrm{Spin}(9)<\mathrm{F}_{4}$, we argue as
follows: $T_{x}\mathbb{O}\mathbb{P}^{2}$ is the spinor representation of
$\mathrm{Spin}(9)$. Under this action
\[
T_{x}\mathbb{O}\mathbb{P}^{2}\supset\mathbb{S}^{15}\cong\mathrm{Spin}%
(9)/\mathrm{Spin}(7)
\]
(see P.283 of \cite{Harvey_spinors}). Hence we can use $\sigma\in
\mathrm{Spin}(9)$ to bring $\frac{1}{2}\frac{\partial}{\partial x_{0}^{1}}$ to
$v_{1}$. $\mathrm{Spin}(7)$ fixes $v_{1}$ and hence acts on $T_{v_{1}%
}\mathbb{S}^{15}$, which splits into the vector representation $V_{7}$ and the
spinor representation of $\mathrm{Spin}(7)$. $\left\{  \sigma\left(  \frac
{1}{2}\frac{\partial}{\partial x_{l}^{1}}\right)  \right\}  _{l=1}^{7}$ and
$\{\mathbb{J}_{l}v_{1}\}_{l=1}^{7}$ form two bases of $V_{7}$ having the same
orientation. Then we can bring $\left\{  \sigma\left(  \frac{1}{2}%
\frac{\partial}{\partial x_{l}^{1}}\right)  \right\}  _{l=1}^{7}$ to
$\{\mathbb{J}_{l}v_{1}\}_{l=1}^{7}$ by an element in $\mathrm{Spin(7)}$.
$\left\{  \sigma\left(  \frac{1}{2}\frac{\partial}{\partial x_{l}^{2}}\right)
\right\}  _{l=1}^{7}$ can be brought to $\{\mathbb{J}_{l}v_{2}\}_{l=0}^{7}$ by
$\mathrm{Spin(7)}$ using similar reasoning, because
\[
\mathrm{Spin(7)}/G_{2}\cong\mathbb{S}^{7}\text{ and }\mathrm{G_{2}%
}/\mathrm{SU}(3)\cong\mathbb{S}^{6}%
\]
and $\mathrm{SU}(3)$ acts transitively on the collection of unitary bases of
$\mathbb{C}^{3}$.

By Lemma \ref{inv of secvar under G}, average second variations of $\xi$ and
$g\cdot$ $\xi$ are the same for all $g\in G$, and hence we may assume
\[
\mathbb{J}_{l}v_{j}=\frac{1}{2}\frac{\partial}{\partial x_{l}^{j}}
\]
so that we can apply Lemma \ref{secform} directly.

For the case $\mathbb{A}=\mathbb{R}^{m}$ in which $\mathbb{A}\mathbb{P}
^{1}=\mathbb{S}^{m}$, Lemma \ref{secform} gives
\[
\left\Vert \,\mathrm{I}\mathrm{I}(\frac{1}{2}\frac{\partial}{\partial x^{j}
},\frac{1}{2}\frac{\partial}{\partial x^{k}})\right\Vert ^{2}=\delta_{jk}
\]
which is the usual formula for the second fundamental form of $\mathbb{S}
^{m}\subset\mathbb{R}^{m+1}$. Together with Theorem \ref{secvarorbit}, the
result of Lawson and Simons \cite{Lawson_stable} is reproduced:
\[
\mathrm{tr}\,\mathcal{Q}_{\xi}=\sum_{j,k=1}^{p,q}(-1)=-pq\leq0
\]
where $p+q=m$, implying that the average second variation of a rectifiable
current of non-zero volume in $\mathbb{S}^{n}$ is negative for $0<p<m$, and
hence cannot be stable.

Now let's turn to the other four cases. Lemma \ref{secform} gives
\[
\Vert\,\mathrm{I}\mathrm{I}(e_{j},n_{k})\Vert^{2}=\left\{
\begin{array}
[c]{ll}%
0 & \text{for }n_{k}=\pm\mathbb{J}_{l}e_{j}\text{ for some $1\leq l\leq
\Lambda$}\\
\frac{1}{4} & \text{otherwise}%
\end{array}
\right.
\]
and
\[
\left\langle \,\mathrm{I}\mathrm{I}(e_{j},e_{j}),\,\mathrm{I}\mathrm{I}
(n_{k},n_{k})\right\rangle =\left\{
\begin{array}
[c]{ll}%
1 & \text{for }n_{k}=\pm\mathbb{J}_{l}e_{j}\text{ for some $1\leq l\leq
\Lambda$}\\
\frac{1}{2} & \text{otherwise}%
\end{array}
\right.
\]
so the summand appeared in Theorem \ref{secvarorbit} is
\[
2\Vert\,\mathrm{I}\mathrm{I}(e_{j},n_{k})\Vert^{2}-\left\langle \,\mathrm{I}
\mathrm{I}(e_{j},e_{j}),\,\mathrm{I}\mathrm{I}(n_{k},n_{k})\right\rangle
=\left\{
\begin{array}
[c]{ll}%
-1 & \text{for }n_{k}=\pm\mathbb{J}_{l}e_{j}\text{ for some $1\leq
l\leq\Lambda$}\\
0 & \text{otherwise}%
\end{array}
\right.
\]
meaning that for each $e_{j}$, every $\mathbb{J}_{l}e_{j}$-direction normal to
$\xi$ contributes $-1$ to $\mathrm{tr}\,\mathcal{Q}_{\xi}$, and all other
normal directions have no effect. Hence
\begin{align*}
\mathrm{tr}\,\mathcal{Q}_{\xi}  &  =-\sum_{j=1}^{p}\,(\text{number of $l$ such
that $\pm\mathbb{J}_{l}e_{j}\not \in B$})\\
&  =-\sum_{j=1}^{p}\sum_{l=1}^{\Lambda}\Vert e_{1}\wedge\ldots\wedge
\mathbb{J}_{l}e_{j}\wedge\ldots\wedge e_{p}\Vert^{2}\\
&  =-\sum_{l=1}^{\Lambda}\Vert\mathbb{J}_{l}\cdot\xi\Vert^{2}\leq0\text{.}%
\end{align*}
(Here $\mathbb{J}$ acts on $\xi$ by Leibniz rule.) Equality holds if and only
if $\Vert\mathbb{J}_{l}\cdot\xi\Vert^{2}=0$ for all $1\leq l\leq\Lambda$,
meaning that $\xi$ is invariant under each $\mathbb{J}_{l}$, and hence
invariant under the $\mathbb{S}^{\Lambda-1}$-family of complex structures.
\ Hence we obtain the following theorem:

\begin{theorem}
\label{main}In $\mathbb{AP}^{n}$, where $\mathbb{A\in}\left\{  \mathbb{R}%
,\mathbb{C},\mathbb{H},\mathbb{O},\mathbb{R}^{m}\right\}  $, any stable
minimal submanifold $S$ (or more generally rectifiable current) must be
complex, by which we means $T_{x}S$ is invariant under all the linear complex
structures at $x$ for almost every $x\in S$.
\end{theorem}

We remark that in $\mathbb{H}\mathbb{P}^{n}$, a quaternionic submanifold must
be totally geodesic.

\section{Appendix: Average second variation in symmetric orbits}

Our aim is to prove the following theorem, which we have used in the last
section to compute the average second variation of the volume of a cycle in
\noindent$\mathbb{A}\mathbb{P}^{n}$ along directions in $\mathfrak{h}
_{\mathbb{A}}\left(  n+1\right)  $:

\vspace{3mm}

\noindent \textbf{Theorem: }{\it Assume that $M=G/K$ is a compact symmetric space which is a $G$-orbit of an
orthogonal representation $\mathbb{V}$ of $G$. The projection of each
$u\in\mathbb{V}$ determines a vector field $V_{u}$, or simply $V$, on $M$. The
average second variation of an oriented orthonormal $p$-frame $\xi=e_{1}%
\wedge\ldots\wedge e_{p}$ at $x\in M$ under all such vector fields is given
by
\[
\mathrm{tr}\,\mathcal{Q}_{\xi}=\sum_{j,k=1}^{p,q}\left(  2\,\Vert
\,\mathrm{I}\mathrm{I}(e_{j},n_{k})\Vert^{2}-\left\langle \,\mathrm{I}%
\mathrm{I}(e_{j},e_{j}),\,\mathrm{I}\mathrm{I}(n_{k},n_{k})\right\rangle
\right)
\]
where $\,\mathrm{I}\mathrm{I}$ is the second fundamental form of
$M\subset\mathbb{V}$ at $x$, and $\{e_{j}\}_{j=1}^{p}\cup\{n_{k}\}_{k=1}^{q}$
is an orthonormal basis of $T_{x}M$.}

\vspace{3mm}

The method of proof is similar to \cite{Lawson_stable}. The Lie algebra
$\mathfrak{g}$ of $G$ decomposes:
\[
\mathfrak{g}=\mathfrak{k}\oplus\mathfrak{m}
\]
where $\mathfrak{m}:=\mathfrak{k}^{\perp}$. On $G$ we have a natural
$G$-invariant metric given by negative of the Killing form, which can be
scaled such that $\mathfrak{m}$ is isometric to $T_{x}M$. We shall use the
same symbol to denote an element of $\mathfrak{g}$, its induced vector field
on $\mathbb{V}$, and the restricted Killing vector field on $M$. Recall that
\begin{equation}
\lbrack g_{1},g_{2}]_{M}=-[g_{1},g_{2}] \label{bracket}%
\end{equation}
where $[\cdot,\cdot]_{M}$ is the Lie bracket for vector fields on $M$, and
$[\cdot,\cdot]$ is the Lie bracket on $\mathfrak{g}$. On the right hand side
$g_{1},g_{2}$ denote elements in $\mathfrak{g}$, while on the left hand side
they denote their induced Killing vector fields on $M$.

Let's complete $\xi=e_{1}\wedge\ldots\wedge e_{p}$ to an orthonormal basis
$\{e_{1},\ldots,e_{p},n_{1},\ldots,n_{q}\}$ of $T_{x}M\cong\mathfrak{m}$, and
further take an orthonormal basis $\{\beta_{1},\ldots,\beta_{r}\}$ of
$\mathfrak{k}$, so that $\{\beta_{1},\ldots,\beta_{r},e_{1},\ldots,e_{p}
,n_{1},\ldots,n_{q}\}$ forms an orthonormal basis of $\mathfrak{g}$.

We now express the projection $V=V_{u}$ of $u\in\mathbb{V}$ in terms of
Killing vector fields induced by $\mathfrak{g}$ on $M$.

\begin{lemma}
\label{exp for V}
\[
V=\sum_{\mu=1}^{r}\left\langle u,\beta_{\mu}\right\rangle \beta_{\mu}
+\sum_{\nu=1}^{p}\left\langle u,e_{\nu}\right\rangle e_{\nu}+\sum_{\gamma
=1}^{q}\left\langle u,n_{\gamma}\right\rangle n_{\gamma}\text{.}
\]

\end{lemma}

\begin{proof}
Denote the basis $\{\beta_{1},\ldots,\beta_{r},e_{1},\ldots,e_{p},n_{1}
,\ldots,n_{q}\}$ of $\mathfrak{g}$ by $A$.

At $x\in M$ the above equation is obvious, because $\beta_{\mu}(x)=0$,\newline
and $\{e_{1},\ldots,e_{p},n_{1},\ldots,n_{q}\}$ forms an orthonormal basis of
$T_{x}M$.

At another point $y\in M$, let $\{\tilde{e_{1}},\ldots,\tilde{e_{p}}%
,\tilde{n_{1}},\ldots,\tilde{n_{q}}\}$ be an orthonormal basis of $T_{y}%
M\cong\mathfrak{m}$, and we complete it to an orthonormal basis
\[
B=\{\tilde{\beta}_{1},\ldots,\tilde{\beta}_{r},\tilde{e}_{1},\ldots,\tilde
{e}_{p},\tilde{n}_{1},\ldots,\tilde{n}_{q}\}
\]
of $\mathfrak{g}$. Both $A,B$ are orthonormal basis of $\mathfrak{g}$, so
$B=AT$, where $T$ is an orthogonal matrix.
\[
V(x)=\sum_{j}\left\langle u,B_{j}\right\rangle B_{j}=\sum_{j}\left\langle
u,A_{k}T_{j}^{k}\right\rangle A_{i}T_{j}^{i}=\sum_{j}\left\langle
u,A_{j}\right\rangle A_{j}%
\]
since $\sum_{j}T_{j}^{k}T_{j}^{i}=\delta^{ki}$.
\end{proof}

\bigskip

\noindent\textbf{Proof to Theorem \ref{secvarorbit}:} From the second
variation formula \eqref{sec_var}, the average second variation is given by
\[
\mathrm{tr}\,\mathcal{Q}_{\xi}=\sum_{u}\left(  \sum_{j=1}^{p}\left\langle
\mathcal{A}_{V}e_{j},e_{j}\right\rangle \right)  ^{2}+2\sum_{u}\sum
_{j=1,k=1}^{p,q}\left(  \left\langle \mathcal{A}_{V}e_{j},n_{k}\right\rangle
\right)  ^{2}+\sum_{u}\sum_{j=1}^{p}\left\langle \mathcal{A}_{V,V}e_{j}%
,e_{j}\right\rangle
\]
where $u$ runs through an orthonormal basis of $\mathbb{V}$, each gives a
vector field $V=V_{u}$ on $M$ by projection. We compute term by term for the
three terms appeared in the above expression.

Recall \cite{Helgason_symmetric} that for a symmetric space,
\[
\nabla_{K_{1}}K_{2}=\frac{1}{2}\,[K_{1},K_{2}]_{M}
\]
for Killing vector fields $K_{1}$ and $K_{2}$ on $M$. Applying this to the
expression of $V$ given in Lemma \ref{exp for V},
\begin{align}
\nabla_{e_{j}}V  &  =\left\langle u,\partial_{e_{j}}\beta_{\mu}\right\rangle
\beta_{\mu}+\frac{1}{2}\,\left\langle u,\beta_{\mu}\right\rangle [e_{j}
,\beta_{\mu}]_{M}+\left\langle u,\partial_{e_{j}}e_{\nu}\right\rangle e_{\nu
}+\frac{1}{2}\,\left\langle u,e_{\nu}\right\rangle [e_{j},e_{\nu}
]_{M}\nonumber\label{dV}\\
&  +\left\langle u,\partial_{e_{j}}n_{\gamma}\right\rangle n_{\gamma}+\frac
{1}{2}\,\left\langle u,n_{\gamma}\right\rangle [e_{j},n_{\gamma}]_{M}%
\end{align}
where $\partial$ is the trivial connection of $\mathbb{V}$, and so
$\partial_{v}$ is the usual directional derivative along $v\in T_{x}
\mathbb{V}\cong\mathbb{V}$. (Recall that $\beta_{\mu}$, $e_{\nu}$, $n_{\gamma
}$ can be regarded as vector fields on $\mathbb{V}$, and so the above
directional derivatives make sense.)

To simplify the above expression at $x$, notice that $\mathfrak{k}$ induces
zero vectors at $x$, and hence $\beta_{\mu}\in\mathfrak{k}$ vanishes at $x $.
Together with equation \eqref{bracket} and the fact that
\begin{equation}
\lbrack\mathfrak{k},\mathfrak{k}]\subset\mathfrak{k},[\mathfrak{k}
,\mathfrak{m}]\subset\mathfrak{m},[\mathfrak{m},\mathfrak{m}]\subset
\mathfrak{k} \label{t_and_p}%
\end{equation}
we have
\[
\nabla_{e_{j}}V(x)=\left\langle u,\partial_{e_{j}}e_{\nu}\right\rangle e_{\nu
}+\left\langle u,\partial_{e_{j}}n_{\gamma}\right\rangle n_{\gamma}
\]
and hence
\[%
\begin{array}
[c]{lllll}%
\left\langle \mathcal{A}_{V}e_{j},e_{j}\right\rangle  & = & \left\langle
\nabla_{e_{j}}V(x),e_{j}\right\rangle  & = & \left\langle u,\partial_{e_{j}
}e_{j}\right\rangle \text{;}\\
\left\langle \mathcal{A}_{V}e_{j},n_{k}\right\rangle  & = & \left\langle
\nabla_{e_{j}}V(x),n_{k}\right\rangle  & = & \left\langle u,\partial_{e_{j}
}n_{k}\right\rangle \text{.}%
\end{array}
\]

The first term $\sum_{u}\left(  \sum_{j=1}^{p}\left\langle \mathcal{A}
_{V}e_{j},e_{j}\right\rangle \right)  ^{2}$ is
\begin{align*}
\sum_{u}\left(  \sum_{j=1}^{p}\left\langle \mathcal{A}_{V}e_{j},e_{j}
\right\rangle \right)  ^{2}  &  =\sum_{u}\sum_{j,k=1}^{p}\left\langle
u,\partial_{e_{j}}e_{j}\right\rangle \left\langle u,\partial_{e_{k}}
e_{k}\right\rangle \\
&  =\sum_{j,k=1}^{p}\left\langle \partial_{e_{j}}e_{j},\partial_{e_{k}}
e_{k}\right\rangle \\
&  =\left\Vert \sum_{j=1}^{p}\,\mathrm{I}\mathrm{I}(e_{j},e_{j})\right\Vert
^{2}%
\end{align*}
where $\partial_{e_{j}}e_{j}=\,\mathrm{I}\mathrm{I}(e_{j},e_{j})$ because
$\nabla_{e_{j}}e_{j}=[e_{j},e_{j}]_{M}/2=0$.

The second term $2\sum_{u}\sum_{j=1,k=1}^{p,q}\left(  \left\langle
\mathcal{A}_{V}e_{j},n_{k}\right\rangle \right)  ^{2}$ is
\begin{align*}
2\sum_{u}\sum_{j=1,k=1}^{p,q}\left(  \left\langle \mathcal{A}_{V}e_{j}
,n_{k}\right\rangle \right)  ^{2}  &  =2\sum_{u}\sum_{j=1,k=1}^{p,q}\left(
\left\langle u,\partial_{e_{j}}n_{k}\right\rangle \right)  ^{2}\\
&  =2\sum_{j,k=1}^{p,q}\Vert\partial_{e_{j}}n_{k}\Vert^{2}\\
&  =2\sum_{j,k=1}^{p,q}\Vert\,\mathrm{I}\mathrm{I}(e_{j},n_{k})\Vert^{2}%
\end{align*}
where $\partial_{e_{j}}n_{k}=\,\mathrm{I}\mathrm{I}(e_{j},n_{k})$ at $x$
because $\nabla_{e_{j}}n_{k}(x)=[e_{j},n_{k}]_{M}/2=0$.

Now we turn to compute the third term $\sum_{u}\sum_{j=1}^{p}\left\langle
\mathcal{A}_{V,V}e_{j},e_{j}\right\rangle $, which is more complicated. At $x
$,
\begin{align*}
\left\langle \mathcal{A}_{V,V}e_{j},e_{j}\right\rangle  &  =\left\langle
\nabla_{V}\nabla_{e_{j}}V-\nabla_{\nabla_{V}e_{j}}V,e_{j}\right\rangle \\
&  =\left\langle \nabla_{V}\nabla_{e_{j}}V,e_{j}\right\rangle \\
&  =\sum_{\nu=1}^{p}\left\langle u,e_{\nu}\right\rangle \left\langle
\nabla_{e_{\nu}}\nabla_{e_{j}}V,e_{j}\right\rangle +\sum_{\gamma=1}
^{q}\left\langle u,n_{\gamma}\right\rangle \left\langle \nabla_{n_{\gamma}
}\nabla_{e_{j}}V,e_{j}\right\rangle
\end{align*}
where $\nabla_{\nabla_{V}e_{j}}V(x)=0$ because
\[
\nabla_{V}e_{j}(x)=\sum_{\nu=1}^{p}\left\langle u,e_{\nu}\right\rangle
\frac{[e_{\nu},e_{j}]_{M}}{2}+\sum_{\gamma=1}^{q}\left\langle u,n_{\gamma
}\right\rangle \frac{[n_{\gamma},e_{j}]_{M}}{2}=0\text{.}
\]

We now compute the first part $\sum\left\langle u,e_{\nu}\right\rangle
\left\langle \nabla_{e_{\nu}}\nabla_{e_{j}}V,e_{j}\right\rangle $ of the third
term. Differentiating equation \eqref{dV} along $e_{\nu}$, we get
\begin{align*}
\nabla_{e_{\nu}}\nabla_{e_{j}}V(x)  &  =\frac{1}{2}\,\left\langle
u,\partial_{e_{j}}\beta_{\mu}\right\rangle [e_{\nu},\beta_{\mu}]_{M}+\frac
{1}{2}\,\left\langle u,\partial_{e_{\nu}}\beta_{\mu}\right\rangle [e_{j}
,\beta_{\mu}]_{M}\\
&  +\left\langle u,\partial_{e_{\nu}}\partial_{e_{j}}e_{\alpha}\right\rangle
e_{\alpha}+\frac{1}{4}\,\left\langle u,e_{\alpha}\right\rangle [e_{\nu}
,[e_{j},e_{\alpha}]_{M}]_{M}\\
&  +\left\langle u,\partial_{e_{\nu}}\partial_{e_{j}}n_{\gamma}\right\rangle
n_{\gamma}+\frac{1}{4}\,\left\langle u,n_{\gamma}\right\rangle [e_{\nu}
,[e_{j},n_{\gamma}]_{M}]_{M}\text{.}%
\end{align*}

Using the identity $\left\langle [X,Y]_{M},Z\right\rangle =-\left\langle
Y,[X,Z]_{M}\right\rangle $ for Killing vector fields $X,Y,Z$, together with
the relation \eqref{t_and_p} repeatedly, we get
\[
\left\langle \nabla_{e_{\nu}}\nabla_{e_{j}}V(x),e_{j}\right\rangle
=\left\langle u,\partial_{e_{\nu}}\partial_{e_{j}}e_{j}\right\rangle
\]
and so
\begin{align}
&  \sum_{u}\sum_{j=1}^{p}\sum_{\nu=1}^{p}\left\langle u,e_{\nu}\right\rangle
\left\langle \nabla_{e_{\nu}}\nabla_{e_{j}}V,e_{j}\right\rangle
\nonumber\label{part1}\\
&  =\sum_{u}\sum_{j=1}^{p}\sum_{\nu=1}^{p}\left\langle u,e_{\nu}\right\rangle
\left\langle u,\partial_{e_{\nu}}\partial_{e_{j}}e_{j}\right\rangle
\nonumber\\
&  =\sum_{j,\nu=1}^{p}\left\langle \partial_{e_{\nu}}\partial_{e_{j}}
e_{j},e_{\nu}\right\rangle \nonumber\\
&  =-\left\Vert \sum_{j=1}^{p}\,\mathrm{I}\mathrm{I}(e_{j},e_{j})\right\Vert
^{2}\text{.}%
\end{align}

Now proceed to compute the second part $\sum\left\langle u,n_{\gamma
}\right\rangle \left\langle \nabla_{n_{\gamma}}\nabla_{e_{j}}V,e_{j}
\right\rangle $ of the third term. Differentiating the equation \eqref{dV}
along $n_{\gamma}$, we get
\begin{align*}
\nabla_{n_{\gamma}}\nabla_{e_{j}}V(x)  &  =\frac{1}{2}\,\left\langle
u,\partial_{e_{j}}\beta_{\mu}\right\rangle [n_{\gamma},\beta_{\mu}]_{M}
+\frac{1}{2}\,\left\langle u,\partial_{n_{\gamma}}\beta_{\mu}\right\rangle
[e_{j},\beta_{\mu}]_{M}\\
&  +\left\langle u,\partial_{n_{\gamma}}\partial_{e_{j}}e_{\nu}\right\rangle
e_{\nu}+\frac{1}{4}\,\left\langle u,e_{\nu}\right\rangle [n_{\gamma}
,[e_{j},e_{\nu}]_{M}]_{M}\\
&  +\left\langle u,\partial_{n_{\gamma}}\partial_{e_{j}}n_{\alpha
}\right\rangle n_{\alpha}+\frac{1}{4}\,\left\langle u,n_{\alpha}\right\rangle
[n_{\gamma},[e_{j},n_{\alpha}]_{M}]_{M}%
\end{align*}
and so
\[
\left\langle \nabla_{n_{\gamma}}\nabla_{e_{j}}V(x),e_{j}\right\rangle
=\left\langle u,\partial_{n_{\gamma}}\partial_{e_{j}}e_{j}\right\rangle
\text{.}
\]

\begin{align}
&  \sum_{u}\sum_{j=1}^{p}\sum_{\gamma=1}^{q}\left\langle u,n_{\gamma
}\right\rangle \left\langle \nabla_{n_{\gamma}}\nabla_{e_{j}}V,e_{j}
\right\rangle \nonumber\label{part2}\\
&  =\sum_{j,\gamma=1}^{p,q}\left\langle \partial_{n_{\gamma}}\partial_{e_{j}
}e_{j},n_{\gamma}\right\rangle \nonumber\\
&  =-\sum_{j,\gamma=1}^{p,q}\left\langle \,\mathrm{I}\mathrm{I}(e_{j}
,e_{j}),\,\mathrm{I}\mathrm{I}(n_{\gamma},n_{\gamma})\right\rangle \text{.}%
\end{align}

Adding up equations \eqref{part1} and \eqref{part2}, we get the third term
\[
-\left\Vert \sum_{j=1}^{p}\,\mathrm{I}\mathrm{I}(e_{j},e_{j})\right\Vert
^{2}-\sum_{j,\gamma=1}^{p,q}\left\langle \,\mathrm{I}\mathrm{I}(e_{j}
,e_{j}),\,\mathrm{I}\mathrm{I}(n_{\gamma},n_{\gamma})\right\rangle \text{.}
\]

Adding up all the three terms, the average second variation is
\[
\sum_{j,k=1}^{p,q}\left(  2\,\Vert\,\mathrm{I}\mathrm{I}(e_{j},n_{k})\Vert
^{2}-\left\langle \,\mathrm{I}\mathrm{I}(e_{j},e_{j}),\,\mathrm{I}
\mathrm{I}(n_{k},n_{k})\right\rangle \right)  \text{.\ \rule{5pt}{5pt}}
\]

\begin{quote}
\textbf{Acknowledgement:} The second author is partially supported by an RGC
grant from the Hong Kong Government.
\end{quote}

\noindent

\bibliographystyle{amsplain}
\bibliography{geometry}

\end{document}